\documentclass[letter, 11pt, twoside]{article}

\usepackage{amsmath, amscd, amsfonts, amssymb, amsthm, latexsym, url, color}
\usepackage{graphicx}
\usepackage[left=1.3in, right=1.2in, top=1in, bottom=0.8in, includefoot, headheight=13.6pt]{geometry}
\input{xy}
\xyoption{all}

\newtheorem{theorem}{Theorem}[section]
\newtheorem{lemma}[theorem]{Lemma}
\newtheorem{corollary}[theorem]{Corollary}
\newtheorem{proposition}[theorem]{Proposition}

\theoremstyle{definition}
\newtheorem{definition}[theorem]{Definition}
\newtheorem{remark}[theorem]{Remark}
\newtheorem{example}[theorem]{Example}

\newcommand{\al}{\alpha}
\newcommand{\bb}{\mathbb}

\newcommand{\comment}[1]{}

\newcommand{\into}{\hookrightarrow}
\newcommand{\isoto}{\stackrel{\simeq}{\to}}

\newcommand{\mult}[1]{#1^{\!\times}}

\newcommand{\onto}{\twoheadrightarrow}
\newcommand{\op}{\operatorname}
\newcommand{\pid}[1]{\langle #1 \rangle}
\newcommand{\res}{\overline}
\newcommand{\roi}{\mathcal{O}}

\newcommand{\To}{\longrightarrow}

\newcommand{\xto}{\xrightarrow}

\renewcommand{\cal}{\mathcal}

\renewcommand{\frak}{\mathfrak}

\renewcommand{\tilde}{\widetilde}
\renewcommand{\Im}{\operatorname{Im}}
\renewcommand{\ker}{\operatorname{Ker}}

\DeclareMathOperator{\Frac}{Frac}

\renewcommand{\pid}[1]{\langle{#1}\rangle}

\usepackage{fancyhdr}

\usepackage{sectsty}
\sectionfont{\Large\sc\centering}
\subsectionfont{\sc}
\chapterfont{\sc\centering}
\chaptertitlefont{\LARGE\centering}
\partfont{\centering}

\begin{document}
\title{$K_2$ of localisations of local rings}

\author{Matthew Morrow}

\date{}

\maketitle

\begin{abstract}
We show that $K_2$ of ``sufficiently regular'' localisations of local rings (e.g.~inverting a sequence of regular parameters) can be described by the Steinberg presentation. The proof is inductive on the number of irreducible elements being inverted, successively using a generalisation of a co-Cartesian square first exploited by Dennis and Stein.
\end{abstract}

\section{Introduction}
The culmination of work by R.~Dennis, M.~Stein, W.~van der Kallen, et al.~\cite{Dennis1973a, Dennis1973, Dennis1974, Stein1973, vanderKallen1971, vanderKallen1975, vanderKallen1977} in the early 1970s showed that $K_2$ of a local ring $A$ may be described via the Steinberg presentation (assuming that the residue field of $A$ has $>5$ elements): \[K_2(A)=\mult A\otimes_{\bb Z}\mult A/\pid{a\otimes 1-a\,|\,a,1-a\in\mult A}.\] In other words, the natural map \[K_2^M(A)\to K_2(A)\tag{\dag}\] is an isomorphism, where $K_*^M$ denotes Milnor $K$-theory. Their results apply more generally to $5$-fold stable rings. The main goal of this paper is to show that the isomorphism (\dag) continues to hold if $A$ is regular and we invert a suitable collection of elements of $A$:

\begin{theorem}
Let $A$ be a regular local ring whose residue field has $>5$ elements, and let $t_1,\dots,t_n\in A$ be irreducible elements with the following property: the quotient of $A$ by an ideal generated by any number of $t_1,\dots,t_n$ is still regular. Then \[K_2^M(A_{t_1\cdots t_n})\to K_2(A_{t_1\cdots t_n})\] is an isomorphism.
\end{theorem}

Note that the ring $A_{t_1\cdots t_n}$ is not $5$-fold stable (we check this in remark \ref{remark_stability_does_not_pass_to_localisation}) and so the classical comparison results do not apply. In fact, as far as the author is aware, this is the first result showing that the Steinberg presentation remains valid for a class of rings outside the $5$-fold stable ones (and arbitrary fields).

The idea of proof behind the main theorem is an induction in which we successively invert $t_1,\dots,t_n$ using localisation sequences for both Milnor and Quillen $K$-theory in low degrees. Owing to the appearance of $K_1$ terms in the localisation sequences, we must first consider an analogue of the main theorem for $K_1$, which is equivalent to giving conditions under which $SK_1$ of a localisation of a local ring vanishes; this is the purpose of section \ref{section_K0_and_K1}. In fact, using only this straightforward result for $K_1$, we show in remark \ref{remark_weaker_version_of_main_theorem} that, in the notation of the main theorem, $K_2^M(A_{t_1\cdots t_n})\to K_2(A_{t_1\cdots t_n})$ is surjective after tensoring by $\bb Q$.

Section \ref{section_review} is then a summary of the classical work by Dennis, Stein, van der Kallen, et al., where we review the notion of $k$-fold stability, describe symbolic elements of $K_2$, and give the basic properties the ``Dennis-Stein-Suslin-Yarosh'' $\rho$ map. A version of this map was first used by Dennis and Stein \cite{Dennis1973a} in their work proving that $K_2$ of a discrete valuation ring embeds into $K_2$ of its field of fractions; it was later extended to higher degree Milnor $K$-theory by S.~Suslin and V.~Yarosh \cite{Suslin1991}.

Section \ref{section_main_part} is the bulk of the proof of the main theorem. The first goal is to establish proposition \ref{proposition_the_cocartesian_square}, a co-Cartesian square generalising ones found in \cite{Dennis1975} and \cite{Suslin1991}. A corollary of this is a localisation sequence for Milnor $K$-theory. In the case of inverting a single prime element of $A$, the main theorem follows almost immediately by comparing this localisation sequence with the one for Quillen $K$-theory. In the general case, to proceed inductively, one must carefully check that, even though $A_{t_1\cdots t_{n-1}}$ is no longer $5$-fold stable, it continues to satisfy certain similar properties. Most important of these is perhaps (A3) on page \pageref{page_A3}, stating that the group of units of $A_{t_1\cdots t_{n-1}}$ which are congruent to $1$ mod $t_n$ is generated by its elements of the form $1+ut_n$ with $u\in\mult{A_{t_1\cdots t_{n-1}}}$.

Throughout the paper we give examples to show that certain assumptions cannot be discarded.

\begin{example}
We finish the introduction by providing a collection of examples to which our main theorem applies:
\begin{enumerate}
\item Let $A$ be a regular, local ring whose residue field has $>5$ elements, and let $t_1,\dots,t_n$ be part of a sequence of regular parameters. Then the theorem holds.
\item Let $k$ be a field with $>5$ elements and let $A$ be the localisation of $k[X,Y]$ at the origin. Let $t_1,\dots,t_n\in A$ be local equations of lines through the origin (i.e.~each $t_i$ is a non-zero linear expression in $X$ and $Y$). Then the theorem applies (and, by passing to the limit, the conclusion of the theorem would remain valid if we inverted the local equations of {\em all} lines through the origin).
\item Let $A$ be a regular, local ring whose residue field has $>5$ elements. Then the ring $A[[t]]$ of formal Taylor series is still such a ring, and $A[[t]],t$ satisfies the conditions of the theorem. Therefore $K_2(A((t)))=K_2^M(A((t)))$, where $A((t))=A[[t]][t^{-1}]$ is the ring of formal Laurent series.
\end{enumerate}
\end{example}

\subsection*{Notation and conventions}
All rings are commutative, unital, and Noetherian. We say that an element $t$ of a ring $R$ is a prime element if and only if it generates a non-zero, proper prime ideal; in a UFD this is the same as an irreducible element. When $t$ is a prime element of a domain $R$, the localisation $R_{tR}$ is a discrete valuation ring and we write $\nu_t$ for the associated $t$-adic valuation.

\subsection*{Acknowledgments}
I am very grateful to M.~Nori for many interesting discussions concerning low degree $K$-theory, especially of irregular rings, during the preparation of this paper.

\section{Preliminary results on $K_0$ and $K_1$}\label{section_K0_and_K1}
We begin with some straightforward results on $K_0$ and $K_1$. The following is classical but I could not find a reference:

\begin{lemma}\label{lemma_K0_under_localisations}
Let $R$ be a regular ring for which $K_0(R)=\bb Z$. Then $K_0(S^{-1}R)=\bb Z$ for any multiplicative system $S\subset R$.
\end{lemma}
\begin{proof}
Let $V$ be a finitely generated projective module over $S^{-1}R$. Pick a finitely generated $R$ module $M$ for which $S^{-1}M=V$; since $R$ is regular, $M$ admits a finite length resolution by finitely generated projective $R$-modules. Each of these projectives is stably free by assumption, and so it is easy to modify the finite projective resolution to obtain a finite free resolution. Base changing to $S^{-1}R$ provides a finite free resolution of $V$, whence it is stably free.
\end{proof}

\comment{
\begin{lemma}
Let $R$ be $1$-fold stable (see definition \ref{definition_k_fold_stable}). Then $SK_1(R)=0$, i.e. $K_1(R)\cong\mult R$.
\end{lemma}
\begin{proof}
$1$-fold stability is the same as the strongest of Bass' stable range condition, called $SR_2(R)$ in \cite[Chap.~V, \S3]{Bass}; theorem 4.1 of loc.~cit.~implies that $GL(R)/E(R)$ is generated by elements from $GL_1(R)$. Hence $\det:K_1(R)\to\mult R$ is an isomorphism.

Of course, Bass' general theorem is not required in this case:~A rather straightforward induction shows that if $a_1,\dots,a_n\in R$ satisfy $\pid{a_1,\dots,a_n}=R$, then there exists $r_2,\dots,r_n\in R$ such that $a_1+r_2a_2+\dots+r_na_n\in\mult R$ (the case $n=2$ is exactly the definition of $1$-fold stability). From this it is easy to transform by row and column operations, i.e.~modulo $E_n(R)$, a general element of $SL_n(R)$ to one of the form \[\begin{pmatrix}u&0&\cdots&0\\0&&&\\\vdots&&g&\\0&&&\end{pmatrix},\] where $u\in\mult R$ and $g\in SL_{n-1}R$. But it is well known that $\begin{pmatrix}u^{-1}&0\\0&u\\\end{pmatrix}\in E_2(R)$ (see, e.g.~\cite[Chap.~V, Prop.~1.7]{Bass} or \cite[Corol.~2.1.3]{Rosenberg}, and multiplying by this shows that our original element of $SL_n(R)$ is equivalent modulo $E_n(R)$ to an element of $SL_{n-1}(R)$. By induction we deduce $SL_n(R)=E_n(R)$.
\end{proof}
}
\begin{remark}\label{remark_preliminary_comments_on_R_t}
Suppose that $R$ is a domain and that $t\in R$ is a prime element. In section \ref{section_main_part} we will be interested in the following assumption:
\begin{description}
\item[\bf(A1)] $R_t\cap R_{tR}=R$.
\end{description}
Assumption (A1) implies that $\mult R_t\cap\mult R_{tR}=\mult R$, and so the $t$-adic valuation induces an exact sequence \[0\to \mult R\to\mult R_t\to\bb Z\to 0.\] In other words, a unit of $R_t$ may be decomposed as $ut^n$ for some unique $u\in\mult R$ and $n\in\bb Z$.

Assumption (A1) holds as soon as $R$ is normal; indeed, in that case, $R=\bigcap_{\frak p} R_{\frak p}$ (where $\frak p$ runs over the height one prime ideals of $R$) and $R_t\subseteq\bigcap_{\frak p\neq tR} R_{\frak p}$.
\end{remark}

\begin{proposition}
Suppose that $R$ is a ring and that $t_1,\dots,t_n\in R$ is a sequence of non-zero-divisor, non-unit elements such that:
\begin{enumerate}
\item $R$ is regular, $K_0(R)=\bb Z$, and $K_1(R)\cong\mult R$.
\item For $i=1,\dots, n$, the ring $R/t_iR$ is regular and $K_0(R/t_iR)=\bb Z$.
\end{enumerate}
Then $K_1(R_{t_1\cdots t_n})=\mult R_{t_1\cdots t_n}$.
\end{proposition}
\begin{proof}
Suppose first that $n=1$; write $t=t_1$. Since $R/tR$ is a regular ring whose spectrum is connected (since $K_0(R/t_iR)=\bb Z$), it is a domain. Therefore $t$ is a prime element of $R$ (which is a domain by the same argument); by the previous remark, there is an exact sequence \[0\to \mult R\to\mult R_t\to\bb Z\to 0.\] Next consider the end of the long exact localisation sequence for $K$-theory: \[K_1(R)\to K_1(R_t)\to K_0(R/tR)\to K_0(R)\to K_0(R_t).\] Since $K_0(R)=\bb Z$, the rightmost arrow is injective (even an isomorphism by lemma \ref{lemma_K0_under_localisations}), and so $K_1(R_t)\to K_0(R/tR)$ is surjective. Comparing our two sequences via the determinant and rank maps yields
\[\xymatrix{
0\ar[r]&\mult R\ar[r]&\mult R_t\ar[r]&\bb Z\ar[r]& 0\\
&K_1(R)\ar[r]\ar[u]&K_1(R_t)\ar[r]\ar[u]&K_0(R/tR)\ar[r]\ar[u]& 0
}\]
By assumption the left and right vertical arrows are isomorphisms, whence the central arrow is as well.

Now consider the general case, by induction. Put $R'=R_{t_1\cdots t_{n-1}}$ and $t=t_n$; so the inductive hypothesis implies $K_1(R')\cong\mult{R'}$.

{\bf Case:} $t_i\in tR$ for some $i=1,\dots,n-1$. Then $t$ is already a unit in $R'$ and so $R'=R_{t_1\cdots t_n}$; there is nothing more to show in this case.

{\bf Case:} $t_i\not\in tR$ for any $i=1,\dots,n-1$. Since $R/tR$ is a regular ring whose spectrum is connected, it is a domain; therefore $t_1\cdots t_{n-1}$ is non-zero in $R/tR$. So $R'$ and $R'/tR'$ are both localisations of regular rings with $K_0=\bb Z$ by non-zero elements; therefore they are both themselves regular rings with $K_0=\bb Z$, by lemma \ref{lemma_K0_under_localisations}. This reduces the question to the case $n=1$, which we have already treated.
\end{proof}

The following corollary is the $K_1$ analogue of our main theorem:

\begin{corollary}\label{corollary_K_1_behaves_under_localisation}
Let $A$ be a regular local ring and let $t_1,\dots,t_n\in A$ be non-zero elements such that $A/t_iA$ is regular for each $i$. Then $K_1(A_{t_1\cdots t_n})=\mult{A_{t_1\cdots t_n}}$.
\end{corollary}
\begin{proof}
This follows from the proposition since, for any local ring, $K_0=\bb Z$ and $SK_1=0$.
\end{proof}

\begin{remark}
To show the necessity of the regularity hypotheses, let $A$ be a regular local ring and let $t\in A$. Arguing as in the proof of the previous proposition yields a commutative diagram with exact rows
\[\xymatrix{
0\ar[r]&\mult A\ar[r]&\mult A_t\ar[r]&\bb Z\ar[r]& 0\\
&K_1(A)\ar[r]\ar[u]^\cong&K_1(A_t)\ar[r]\ar[u]&G_0(A/tA)\ar[r]\ar[u]& 0
}\]
If $A/tA$ is not regular then there is a finitely generated $A/tA$ module which has infinite projective dimension (e.g.~the residue field); then $G_0(A/tA)\to\bb Z$ is not an isomorphism, and so $K_1(A_t)\to\mult{A_t}$ is also not an isomorphism.
\end{remark}

\begin{remark}\label{remark_weaker_version_of_main_theorem}
From the corollary it is easy to deduce a statement weaker than the main theorem: If $A,t_1,\dots,t_n$ satisfy the conditions of the main theorem, then $K_2(A_{t_1\cdots t_n})_{\bb Q}$ is generated by Steinberg symbols. In other words, $K_2(A_{t_1\cdots t_n})_{\bb Q}=K_2^{(2)}(A_{t_1\cdots t_n})_{\bb Q}$.

If $n=0$ then this follows from theorem \ref{theorem_Milor_equals_Quillen_for_5_fold_stable} below. Then we proceed by induction; put $R=A_{t_1\cdots t_{n-1}}$ and $t=t_n$. If $t\in\mult R$ there is nothing more to show, so suppose not. Consider the localisation sequence for $R\to R_t$, as well as its restriction to part of the Adams decomposition:
\[\xymatrix{
K_2(R)_{\bb Q}\ar[r]& K_2(R_t)_{\bb Q}\ar[r]&K_1(R/tR)_{\bb Q}&\\
K_2^{(2)}(R)_{\bb Q}\ar[r]\ar[u]&K_2^{(2)}(R_t)_{\bb Q}\ar[r]\ar[u]&K_1^{(1)}(R/tR)_{\bb Q}\ar[u]\ar[r] & K_1^{(2)}(R)_{\bb Q}
}\]
The previous corollary may be applied to $A/tA$, with elements $t_1,\dots,t_{n-1}$ mod $tA$, to deduce that $K_1(R/tR)=\mult{(R/tR)}$; i.e.~the right vertical map is an isomorphism. Secondly, the previous corollary also implies that $K_1(R)=\mult R$ and so $K_1^{(2)}(R)_{\bb Q}=0$; therefore the central map on the bottom of the diagram is surjective. By the inductive hypothesis, the left vertical map is an isomorphism.

From a diagram chase we now see that the central vertical map is an isomorphism (one only needs to check it is surjective) completing the proof.
\end{remark}

\section{Symbolic descriptions of $K_2$}\label{section_review}
Here we review basic properties of $K_2$ and its presentations by Steinberg and Dennis-Stein symbols, as well as discussing the Dennis-Stein-Suslin-Yarosh $\rho$ map.

\subsection{Rings with lots of units}
There are various notions of when a ring has a lot of units; we will only need the following, which is a classical condition under which ``general-position'' type arguments work well:

\begin{definition}\label{definition_k_fold_stable}
Let $k\ge1$ be an integer. A ring $R$ is said to be {\em $k$-fold stable} if and only if whenever $a_1,b_1,\dots,a_k,b_k\in R$ are given such that $\pid{a_1,b_1}=\dots=\pid{a_k,b_k}=R$, then there exists $r\in R$ such that $a_1+rb_1,\dots,a_k+rb_k$ are units.

We will say that $R$ is {\em weakly $k$-fold stable} if and only if whenever $b_1,\dots,b_{k-1}\in R$ are given, then there exists $u\in\mult R$ such that $1+ub_1,\dots,1+ub_{k-1}$ are units. Note that weak $k$-fold stability follows from $k$-fold stability (using the pairs $\pid{1,b_1},\dots,\pid{1,b_{k-1}},\pid{0,1}$).
\end{definition}

\begin{remark}\label{remark_1_fold_stable}
In other words, $R$ is $1$-fold stable if and only if $\mult R\to(R/aR)^\times$ is surjective for all $a\in R$.
\end{remark}

\begin{remark}
A semi-local ring is $k$-fold stable if and only if all of its residue fields have $>k$ elements.
\end{remark}

The notion of weak $k$-fold stability does not appear anywhere in the literature; we introduce it only to be able to clearly state the following and lemma \ref{lemma_skew_symmetry_from_other_relations}:

\begin{lemma}\label{lemma_weak_stability}
If $R$ is weakly $k$-fold stable then so is $S^{-1}R$ for any multiplicative system $S\subset R$.
\end{lemma}
\begin{proof}
Let $b_1,\dots,b_{k-1}\in R$ be given, and pick $s\in S$ such that each $b_i$ may be written as $b_i=b_i's^{-1}$ for some $b_i'\in R$. By the weak $k$-fold stability of $R$, there exists $u\in\mult R$ such that $1+ub_1',\dots,1+ub_{k-1}'\in\mult R$. Replace $u$ by $us$ to complete the proof.
\end{proof}

Our main result would follow from theorem \ref{theorem_Milor_equals_Quillen_for_5_fold_stable} if it were the case that localisations of $k$-fold stable rings remained $k$-fold stable. The following example is provided to show that this is not the case; although we work with a specific example, the proof works in general:

\begin{remark}\label{remark_stability_does_not_pass_to_localisation}
Let $\roi$ be a discrete valuation ring with residue field $K$, and put $R=\roi[[t]]$. Then $R$ is $k$-fold stable for every $k<|K|$, but we will show that $R_t=\roi((t))$ is not even $1$-fold stable.

Let $\pi\in\roi$ be a uniformiser. Since $\pid{\pi,1+\pi^2t^{-1}}=\roi((t))$, it is sufficient to show that there exists no $c\in \roi((t))$ satisfying $\pi+c(1+\pi^2t^{-1})\in\mult{\roi((t))}$. For a contradiction, suppose such a $c$ were to exist, and write $c=t^nc_0$ with $n\in\bb Z$, and $c_0\in \roi[[t]]$ not divisible by $t$. So, \[v:=\pi+t^nc_0(1+\pi^2t^{-1})\in\mult{\roi((t))}.\] Write $\nu_t$ for the $t$-adic valuation on $\Frac(\roi[[t]])$; note that $vt^{-\nu_t(v)}\in\mult{\roi[[t]]}$ by remark \ref{remark_preliminary_comments_on_R_t}.

{\bf Case: $n\le0$}. Then $\nu_t(v)=n-1$, so $vt^{1-n}\in\mult{\roi[[t]]}$. But \[vt^{1-n}=\pi t^{1-n}+c_0(t+\pi^2),\] which belongs to the maximal ideal of $\roi[[t]]$ so is certainly not a unit.

{\bf Case: $n\ge1$.} Then $\nu_t(v)=0$, so $v\in\mult A$, which is again absurd.

This completes the proof that $\roi$ is not even $1$-fold stable. Yet, if $|K|>5$, then our main theorem will imply that $K_2(\roi((t)))$ is described by the Steinberg presentation.
\end{remark}

\subsection{Steinberg symbols}
Given $a,b\in\mult R$, the corresponding Steinberg symbol of $K_2(R)$ will be written $\{a,b\}$ as usual. These symbols \cite[\S8]{Milnor1970} satisfy the following relations in $K_2(R)$:
\begin{enumerate}
\item[(S1)] Bilinearity: $\{a,bc\}=\{a,b\}+\{a,c\}$ and $\{ac,b\}=\{a,b\}+\{c,b\}$ for $a,b,c\in\mult R$.
\item[(S2)] Skew-symmetry: $\{a,-a\}=0$ for $a\in\mult R$.
\item[(S3)] Steinberg relation: $\{a,1-a\}=0$ for $a\in\mult R$ such that $1-a\in\mult R$.
\end{enumerate}
Let $K_2^M(R)$ be the abelian group generated by symbols $\{a,b\}$, for $a,b\in\mult R$, subject to relations (S1) and (S3) (we will say a word about (S2) in a moment); this is the second Milnor $K$-group. Thus there is a natural homomorphism of abelian groups $K_2^M(R)\to K_2(R),\;\{a,b\}\mapsto\{a,b\}$, whose image is the subgroup of $K_2(R)$ generated by the Steinberg symbols.

\begin{lemma}\label{lemma_skew_symmetry_from_other_relations}
Suppose $R$ is weakly $5$-fold stable (e.g.~a localisation of a local ring with $>5$ elements in the residue field). Then relation (S2) follows from (S1) and (S3).
\end{lemma}
\begin{proof}
Suppose $a\in\mult R$; we must show that $\{a,-a\}=0$ in $K_2^M(R)$. The proof is well-known (e.g.~\cite{Kerz2009a}), but it is important to take care to use only the weak version of $k$-fold stability:

Firstly, if actually $1-a\in\mult R$, then use the identity $-a(1-a^{-1})=1-a$ to deduce that $\{a,-a\}=\{a,1-a\}-\{a,1-a^{-1}\}=0$.

Secondly suppose that $s\in\mult R$ happens to satisfy $1-s,1-sa\in\mult R$; then  the first part of the proof tells us that $\{as,-as\}=\{s,-s\}=0$. Therefore,
\begin{align*}
0&=\{as,-as\}\\
&=\{a,-a\}+\{a,-s\}+\{s,-a\}+\{s,-s\}+\{a,-1\}+\{s,-1\}
\end{align*}
i.e.~$\{a,-a\}=-\{a,s\}-\{s,a\}$.

Thirdly, apply weak $3$ fold stability to $-1,-a$ to see that there exists an element $s_1\in \mult R$ with the above properties. Next apply weak $5$-fold stability to $-1,-a,-s_1,-s_1a$ to find $s_2\in\mult R$ which not only satisfies the above properties, but such that $s=s_1s_2$ {\em also} satisfies the above properties. So, \[\{a,-a\}=-\{a,s_1\}-\{s_1,a\},\] \[\{a,-a\}=-\{a,s_2\}-\{s_2,a\},\] and \[\{a,-a\}=-\{a,s_1s_2\}-\{s_1s_2,a\}.\] It follows that $\{a,-a\}=0$.
\end{proof}

The main classical theorem concerning the presentation of $K_2$ by Steinberg symbols is the following, due to W.~van der Kallen, H.~Maazen, and J.~Stienstra \cite{vanderKallen1975} \cite[\S8]{vanderKallen1977}:

\begin{theorem}\label{theorem_Milor_equals_Quillen_for_5_fold_stable}
If $R$ is $5$-fold stable, then $K_2(R)$ is generated by the Steinberg symbols subject to relations (S1) and (S3); in other words, $K_2^M(R)\to K_2(R)$ is an isomorphism.
\end{theorem}

\begin{remark}
We will never need the proper definition of the Steinberg symbol, but include it here for completeness. Recall that $K_2(R)$ may be defined as \[K_2(R)=\ker(St(R)\To E(R)),\] where $St(R)$ is the Steinberg group and $E(R)$ is the group of infinite elementary matrices over $R$. The typical generators of $St(R)$ are \[x_{ij}(a)\qquad(a\in R,\,i,j\ge0,\,i\neq j).\] Put $w_{ij}(a)=x_{ij}(a)x_{ji}(-a^{-1})x_{ij}(a)$ and $h_{ij}(a)=w_{ij}(a)w_{ij}(-1)$.

In this notation, the Steinberg symbol is given by \[\{a,b\}:=[h_{12}(a),h_{13}(b)]\in K_2(R)\] whenever $a,b\in\mult R$.

We are following the modern convention of writing $K_2(R)$ additively, especially when performing symbolic manipulations in it, even though it is a subgroup of the non-abelian group $St(R)$ where we must use multiplicative notation; this should not lead to confusion since quite soon all manipulations will be via symbols.
\end{remark}

\subsection{Dennis-Stein symbols} Given $a,b\in R$ such that $1+ab\in\mult R$, Dennis and Stein \cite{Dennis1973a} defined an element of $K_2(R)$ denoted $\pid{a,b}$. These symbols satisfy the following:
\begin{enumerate}
\item[(D1)] $\pid{a,b}=-\pid{-b,-a}$ for $a,b\in R$ such that $1+ab\in\mult R$.
\item[(D2)] $\pid{a,b}+\pid{a,c}=\pid{a,b+c+abc}$ for $a,b,c\in R$ such that $1+ab,1+ac\in\mult R$.
\item[(D3)] $\pid{a,bc}=\pid{ab,c}+\pid{ac,b}$ for $a,b,c\in R$ such that $1+abc\in\mult R.$
\end{enumerate}

Moreover, any Steinberg symbol $\{a,b\}$, where $a,b\in\mult R$, is equal to a Dennis-Stein symbol in $K_2(R)$: \[\{a,b\}=\pid{(a-1)b^{-1},b}.\] It follows that the Dennis-Stein symbol $\pid{a,b}$ can be expressed as a Steinberg symbol whenever $a$ or $b$ is a unit: \[\pid{a,b}=\begin{cases}\{-a,1+ab\}&\mbox{if }a\in\mult R,\\\{1+ab,b\}&\mbox{if }b\in\mult R.\end{cases}\]

The main classical theorem concerning the presentation of $K_2$ by Dennis-Stein symbols is the following, again due to W.~van der Kallen, H.~Maazen, and J.~Stienstra \cite{vanderKallen1975, vanderKallen1977}: 

\begin{theorem}
If $R$ is local or $3$-fold stable, then $K_2(R)$ is the abelian group generated by the Dennis-Stein symbols subject to relations (D1)--(D3).
\end{theorem}

\begin{remark}
To be precise, the Dennis-Stein symbol is defined to be
\[\pid{a,b}=x_{21}\left(\frac{-b}{1+ab}\right)x_{12}(a)x_{21}(b) x_{12}\left(\frac{-a}{1+ab}\right)h_{12}(1+ab)\in K_2(R),\]
though we will never need this formula.
\end{remark}

\subsection{The Dennis-Stein-Suslin-Yarosh $\rho$ map}\label{section_Dennis_Stein_map}
Dennis and Stein proved in \cite{Dennis1975} that if $\roi_K$ is a discrete valuation ring with field of fractions $K$, then $K_2(\roi_K)\to K_2(K)$ is injective. They did this by constructing a certain homomorphism $\rho:1+t\roi_K\to K_2(\roi_K)$ (where $t\in\roi_K$ is a uniformiser) such that the image of $\rho(u)$ in $K_2(K)$ is $\{u,t\}$. The map was extended to the higher degree Milnor $K$-theory of $\roi_K$ by Suslin and Yarosh in \cite{Suslin1991}. 

Dennis and Stein required rather intricate usage of their so-called $(s,t)$-identities between Steinberg symbols (for example, to prove the following two results) since, at the time, they did not know that they could describe $K_2(\roi_K)$ using their Dennis-Stein symbols. With this hindsight, it is not difficult to give a much easier and more general definition of their $\rho$ map; the payoff is the co-Cartesian square of proposition \ref{proposition_the_cocartesian_square}, a version of which appeared in their work. 

Let $R$ be a ring. For $t\in R$ a non-zero divisor, write $\mult{(1+tR)}=(1+tR)\cap\mult R$, i.e.~the elements of $1+tR$ which are units in $R$; note that $\mult{(1+tR)}$ is a group. Define \[\rho_t:\mult{(1+tR)}\to K_2(R),\quad 1+ta\mapsto\pid{a,t},\] where the right hand side is a Dennis-Stein symbol.

\begin{lemma}\label{lemma_properties_of_DS_map}
$\rho_t$ is a homomorphism of groups. Moreover, if $s\in R$ is another zero-divisor and $f\in \mult{(1+stR)}$, then \[\rho_{st}(f)=\rho_s(f)+\rho_t(f).\]
\end{lemma}
\begin{proof}
Relations (D1) and (D2) show that $\rho_t$ is a homomorphism. Writing $f=1+sta$ for some $a\in R$, the second claim is equivalent to \[\pid{a,st}=\pid{as,t}+\pid{at,s},\] which is immediate from relation (D3).
\end{proof}

\begin{corollary}\label{corollary_identity_for_DS_map}
Suppose $u\in\mult R$ is such that $1+t^lu\in\mult R$ for some $l\ge1$; then \[l\rho_t(1+t^lu)=\{-u,1+t^lu\}\] in $K_2(R)$.
\end{corollary}
\begin{proof}
By the previous lemma, \[l\rho_t(1+t^lu)=\rho_{t^l}(1+t^lu)=\pid{u,t^l},\] which equals $\{-u,1+t^lu\}$ by the comparison between Dennis-Stein and Steinberg symbols.
\end{proof}

\begin{remark}
Suppose that $(1+tR)^\times$ is generated by its elements of the form $1+wt$ with $w\in\mult R$. For example, this is true if $R$ is $3$-fold stable (see (A3) on page \ref{page_A3}), or if $R$ is local and $t\in\frak m_R$. Then the comparison between Steinberg and Dennis-Stein symbols implies that the homomorphism $\rho_t$ is characterised by the fact that \[\rho_t(1+wt)=\{-w,1+wt\}\] for all $w\in\mult R$ for which $1+wt\in\mult R$.

In the case of a discrete valuation ring with $t$ being a uniformiser, Dennis and Stein started with this as the definition of $\rho_t$ on such elements, and proved that it extended to a homomorphism under the rule \[\rho_t(1+at)=\left\{-\frac{1+a}{1-t},\frac{1+at}{1-t}\right\}\tag{\dag}\] whenever $a$ is in the maximal ideal. However, their proof works whenever $R$ is local and $t\in\frak m_R$. So we can therefore conclude that if $R$ is a local domain and $t,a\in\frak m_R$, then formula (\dag) remains valid under our definition of $\rho_t$; i.e., $\pid{a,t}=\left\{-\frac{1+a}{1-t},\frac{1+at}{1-t}\right\}$.
\end{remark}

\section{The main calculations}\label{section_main_part}
In this section we prove the main result. First we will construct, under a number of assumptions, a short exact sequence (see corollary \ref{corollary_the_ses}), which serves as a $K_2^M$-analogue of the sequence in remark \ref{remark_preliminary_comments_on_R_t}. From this, the proof that Milnor and Quillen $K_2$ coincide will proceed by a slightly tricky induction; to keep this as clear as possible, we carefully label each assumption as it appears.

Let $R$ be a domain and suppose that $t\in R$ is a prime element of $R$. Write $\nu=\nu_{t}$ for the $t$-adic discrete valuation on $\Frac R$, with ring of integers $R_{tR}$. To start, we make the following assumption, which we briefly discussed in remark \ref{remark_preliminary_comments_on_R_t}:
\begin{description}
\item[\bf(A1)] Assume that $R_t\cap R_{tR}=R$. (E.g.~This is true if $R$ is normal; see the aforementioned remark.)
\end{description}
The assumption implies that any unit $f$ in $R_t$ may be written as $f=ut^n$ for some unique $u\in\mult R$ and $n\in\bb Z$; indeed, $n=\nu(f)$ and $u=ft^{-\nu(f)}$. Supposing that $g=vt^m$ is another unit of $R_t$, written in the same way, the tame symbol $c(f,g)$ is defined in the usual way as \[c(f,g):=(-1)^{\nu(f)\nu(g)} f^{\nu(g)}g^{-\nu(f)}=(-1)^{nm}u^mv^{-n}\in\mult R.\] 

\begin{lemma}
Let $R,t$ be as above, satisfying (A1). If $f,g\in\mult{R_t}$ satisfy $f+g=1$, then $c(f,g)\in\mult{(1+tR)}$ and \[\rho_t(c(f,g))=-\{ft^{-\nu(f)},gt^{-\nu(g)}\}\] in $K_2(R)$, where the left side is the Dennis-Stein-Suslin-Yarosh map from section \ref{section_Dennis_Stein_map}.
\end{lemma}
\begin{proof}
Write $f=ut^n$, $g=vt^m$ as above; we must show that $c(f,g)\in 1+tR$ (we already know it is a unit) and that $\rho_t(c(f,g))=-\{u,v\}$. The proof is a simple case-by-case analysis:
\begin{description}
\item[Case: $m>0$.] Then $n=0$ and $u=1-vt^m$. So $c(f,g)=u^m\in1+t^mR$ and \[\rho_t(c(f,g))=m\rho_t(1-vt^m)\] as $\rho_t$ is a homomorphism (lemma \ref{lemma_properties_of_DS_map}). But corollary \ref{corollary_identity_for_DS_map} implies $m\rho_t(1-vt^m)=\{v,1-vt^m\}=\{v,u\}$, as required.
\item[Case: $n>0$.] Proceed as in the previous case.
\item[Case: $n=m=0$.] Then $u+v=1$, whence $\{u,v\}=0$, and secondly $c(f,g)=1$. So the claim is trivial.
\item[Case: $m<0$.] Then $n=m$ (set $l=-n$ for clarity), $u+v=t^l$, and \[c(f,g)=(-1)^lu^{-l}v^l=(1-u^{-1}t^l)^l\in 1+t^lR.\] Using the fact that $\rho_t$ is a homomorphism and corollary \ref{corollary_identity_for_DS_map}, we see that  \[\rho_t(c(f,g))=l\rho_t(1-u^{-1}t^l)=\{u^{-1},1-u^{-1}t^l\}=\{u^{-1},-vu^{-1}\}.\] But this symbol equals $\{u^{-1},v\}=-\{u,v\}$, as required.\qedhere
\end{description}
\end{proof}

Before the next proposition, we must impose another two assumptions:
\begin{description}
\item[(A2)] Assume that $R$ is weakly $5$-fold stable and that $K_2^M(R)\to K_2(R)$ is an isomorphism. (E.g.~This is true if $R$ is $5$-fold stable.)
\item[(A3)] \label{page_A3} Assume that the group $\mult{(1+tR)}$ is generated by its elements $1+tw$ satisfying the extra condition that $w\in\mult R$. (E.g.~This is true if $R$ is $3$-fold stable: Let $a\in R$ be such that $1+at$ is a unit. Applying $3$-fold stability to $\pid{a,-1},\pid{1,t},\pid{0,1}$ supplies us with a unit $v\in\mult R$ for which $a-v$ and $1+vt$ are also units. Put $w=\frac{a-v}{1+tv}\in\mult R$. Then $(1+vt)(1+wt)=1+at$, which implies that $1+wt$ is a unit and completes the proof.)
\end{description}
Under assumption (A2) we freely identify the two $K_2$-groups and even allow ourselves to think of Dennis-Stein symbols as elements of $K_2^M(R)$; this assumption also implies that $R_t$ is weakly $5$-fold stable (lemma \ref{lemma_weak_stability}) and therefore that $K_2^M(R_t)$ satisfies skew-symmetry (lemma \ref{lemma_skew_symmetry_from_other_relations}), which we will use without mention.

\begin{proposition}[c.f.~\cite{Dennis1975} \cite{Suslin1991}]\label{proposition_the_cocartesian_square}
Let $R,t$ be as above, satisfying (A1) -- (A3). Then
\[\xymatrix{
\mult{(1+tR)}\ar[r]^j\ar[d]_{\rho_t}&\mult R\ar[d]^{\{\cdot,t\}}\\
K_2^M(R)\ar[r] &K_2^M(R_t)
}\]
is a co-Cartesian square of abelian groups.
\end{proposition}
\begin{proof}
Since both maps from $\mult{(1+tR)}$ to $K_2^M(R_t)$ are homomorphisms, it is enough, by (A3), to check the commutativity of the diagram on elements of $\mult{(1+tR)}$ having the form $1+wt$ for some $w\in\mult R$. For such an element, $\rho_t(1+tw)=\pid{w,t}=\{-w,1+wt\}$ in $K_2^M(R)$; the image of this in $K_2^M(R_t)$ is \[\{-w,1+wt\}=\{-wt,1+wt\}-\{t,1+wt\}=0+\{1+wt,t\},\] as required.

Given $f=ut^n,\;g=vt^m\in \mult R_t$, written as above, we have, in $K_2^M(R_t)$, \[\{f,g\}=\{u,v\}+\{u^m,t\}+\{t,v^n\}+\{t,t\}=\{u,v\}+\{c(f,g),t\}.\] This shows that \[K_2^M(R)\oplus\mult R\to K_2^M(R_t),\quad (\{u,v\},c)\mapsto \{u,v\}+\{c,t\}\] is surjective.

Let $\Delta:\mult{(1+tR)}\to K_2^M(R)\oplus\mult R$ be the map $(-\rho_t,j)$, so that $X=(K_2^M(R)\oplus\mult R)/\Im\Delta$ is the pushout which we wish to show is equal to $K_2^M(R_t)$. The previous lemma shows that the homomorphism \[K_2^M(R_t)\to X,\quad \{f,g\}\mapsto(\{ft^{-\nu(f)},gt^{-\nu(g)}\},c(f,g))\mod{\Im\Delta}\] is well-defined. But we have just shown that the natural map $X\to K_2^M(R_t)$ is surjective, and the reader can easily check that $X\to K_2^M(R_t)\to X$ is the identity, thereby completing the proof.
\end{proof}

\begin{corollary}\label{corollary_the_ses}
Let $R,t$ be as above, satisfying (A1) -- (A3). Then there is a short exact sequence \[0\to K_2^M(R)\to K_2^M(R_t)\xto{\res c}\mult R/\mult{(1+tR)}\to 0,\] where $\res c(\{f,g\}):=c(\{f,g\})$ mod $tR$.
\end{corollary}

\begin{example}\label{example_the_ses}
\begin{enumerate}
\item Suppose that $\roi_K$ is a discrete valuation ring whose residue field $k$ has $>5$ elements; let $K$ be its fraction field and $t\in \roi_K$ a uniformiser. Then the pair $\roi_K,t$ satisfy assumptions (A1) -- (A3) and, moreover, $K_2^M(F)\cong K_2(F)$ by Mastumoto's theorem \cite[\S12]{Milnor1970}. We deduce that \[0\to K_2(\roi_K)\to K_2(K)\to \mult k\to0\] is exact, which is the main result of \cite{Dennis1975}.
\item Suppose that $A$ is a local domain whose residue field has $>5$ elements. Then the pair $A[[t]],t$ satisfies (A1) -- (A3) and so the sequence \[0\to K_2^M(A[[t]])\to K_2^M(A((t)))\to\mult A\to 0\] is exact.
\item Suppose that $A$ is a normal, local ring whose residue field $k$ has $>5$ elements, and let $t\in A$ be any prime element. Then the pair $A,t$ satisfies (A1) -- (A3) and so the sequence \[0\to K_2^M(A)\to K_2^M(A_t)\to\mult{(A/tA)}\to 0\] is exact.
\end{enumerate}
\end{example}

The key inductive step of our main proof is contained in the following proposition, in which we give conditions under which the short exact sequence of the corollary forces $K_2^M(R_t)\to K_2(R_t)$ to be an isomorphism.

\begin{proposition}
Let $R,t$ be as above, satisfying (A1) -- (A3). Suppose further that $R$ and $R/tR$ are regular, and that
\begin{description}
\item[(A4)] $\mult R\to\mult{(R/tR)}$ is surjective. (E.g.~This is true if $R$ is $1$-fold stable by remark \ref{remark_1_fold_stable}.)
\item[(A5)] $K_1(R/tR)=\mult{(R/tR)}$ and $K_2(R/tR)=K_2^M(R/tR)$.
\end{description}
Then $K_2^M(R_t)\to K_2(R_t)$ is an isomorphism.
\end{proposition}
\begin{proof}
Let $G_*$ denote the $K$-theory of the category of finitely generated modules over a ring. The following standard argument \cite{vanderKallen1976} shows that $G_*(R)\to G_*(R/tR)\to G_*(R)$ is zero: it is represented by $-\otimes_RR/tR:R\op{-Mod}\to R\op{-Mod}$ and the class of $R/tR$ in $G_0(R)$ is trivial. Thus $K_*(R)\to K_*(R/tR)\to K_*(R)$ is zero. Assumption (A4) and the first part of (A5) implies that $\mult R\to K_1(R)\to K_1(R/tR)=(R/tR)^\times$ is surjective, and so it follows that $K_1(R/tR)\to K_1(R)$ is zero. Similarly, assumption (A4) and the second part of (A5) imply that $K_2(R)\to K_2(R/tR)$ is surjective, and so $K_2(R/tR)\to K_2(R)$ is zero.

The localisation sequence for $K$-theory therefore produces the short exact sequence \[0\to K_2(R)\to K_2(R_t)\to K_1(R/tR)\to 0.\] We compare this with the short exact sequence of the previous corollary to get the commutative diagram
\[\xymatrix{
0\ar[r]&K_2^M(R)\ar[r]\ar[d]&K_2^M(R_t)\ar[r]\ar[d]&\mult R/\mult{(1+tR)}\ar[r]\ar[d]& 0\\
0\ar[r]&K_2(R)\ar[r]&K_2(R_t)\ar[r]&K_1(R/tR)\ar[r]& 0
}\]
The left vertical arrow is an isomorphism by (A2), and the right vertical arrow is an isomorphism by (A4) and the first part of (A5). The proof is complete.
\end{proof}

Except for a small lemma which we defer for a moment, we have reached the main theorem:

\begin{theorem}
Let $A$ be a regular local ring whose residue field has $>5$ elements, and let $t_1,\dots,t_n\in A$ be irreducible elements with the following property: the quotient of $A$ by an ideal generated by any number of $t_1,\dots,t_n$ is still regular. Then \[K_2^M(A_{t_1\cdots t_n})\to K_2(A_{t_1\cdots t_n})\] is an isomorphism.
\end{theorem}
\begin{proof}
To avoid confusing the exposition with the special case $n=1$, we quickly deal with it now; put $t=t_1$ and note that $A/tA$ is regular. We will be done as soon as we verify that $A,t$ satisfy all the other conditions of the previous proposition: (A1) holds because $A$ is regular, hence normal; (A2) holds because $A$ is $5$-fold stable; (A3) holds because $A$ is $3$-fold stable; (A4) holds because $A$ is $1$-fold stable; the first part of (A5) holds because $A/tA$ is a local ring, and the second part holds because $A/tA$ is $5$-fold stable. This concludes the proof in the case $n=1$.

The remainder of the proof is by induction on $n>1$ using the previous results. We may obviously assume that the $t_1,\dots,t_n$ are pairwise non-associated. Let $R=A_{t_1\cdots t_{n-1}}$ and put $t=t_n$. We will show that the pair $R,t$ satisfies all the conditions of the previous proposition, from which the result then follows.

(A1): By the assumption that $t_1,\dots,t_n$ were pairwise non-associated, we see that $t$ is a prime element of $R$; moreover, $R$ is normal. So the pair $R,t$ satisfy assumption (A1).

(A2): By the inductive hypothesis, $K_2^M(R)\to K_2(R)$ is an isomorphism. Moreover, $R$ is weakly $5$-fold stable by lemma \ref{lemma_weak_stability}. So the pair $R,t$ satisfy (A2).

(A3): This will be covered by the next lemma.

(A4): Set $\res A=A/tA$, $\res R=R/tR$, and let $\res t_i$ denote the image of $t_i$ in $A/tA$, for $i=1,\dots,n-1$; so $\res R=\res A_{\res t_1\cdots\res t_{n-1}}$. Then $\res A$ is regular. Also, for each $i=1,\dots,n-1$, the element $\res t_i$ is non-zero in $\res A$ and $\res A/\res t_i\res A=A/\pid{t_i,t}$ is a regular local ring, hence a domain; so $\res t_i$ is a prime element of $\res A$. By repeatedly applying the comments immediately after the introduction of (A1) at the start of this section, we see that a unit $w$ of $\res R$ may be written as $u\res t_1^{\al_1}\dots \res t_{n-1}^{\al_{n-1}}$ for some $u\in\mult{\res A}$ and $\al_1,\dots,\al_{n-1}\in\bb Z$ (the representation might not be unique because some of the $\res t_1,\dots,\res t_{n-1}$ may be associated to one another). Since $A$ is local, there is a unit $\tilde u\in\mult A$ sitting over $u$; then $\tilde u t_1^{\al_1}\dots t_{n-1}^{\al_{n-1}}\in\mult R$ sits over $w$. This proves that $R,t$ satisfies (A4).

(A5)  The quotient of $\res A$ by the ideal generated by any number of $\res t_1,\dots,\res t_{n-1}$ is still regular, so the inductive hypothesis implies that $K_2^M(\res R)\to K_2(\res R)$ is an isomorphism, while corollary \ref{corollary_K_1_behaves_under_localisation} implies that $K_1(\res R)=\mult{\res R}$. So condition (A5) is satisfied for the pair $R,t$.

Finally, note that $R$ and $R/tR=\res R$ are localisations of regular rings, hence are regular. So the pair $R,t$ satisfy all the required conditions to apply the previous proposition, except possibly (A3), which is rather subtle and which we deal with in the next lemma.
\end{proof}

\begin{lemma}
Let $A,t_1,\dots,t_n$ satisfy the conditions of the theorem, with $t_1,\dots,t_n$ being pairwise non-associated; put $R=A_{t_1\cdots t_{n-1}}$ and $t=t_n$. Then the pair $R,t$ satisfy assumption (A3).
\end{lemma}
\begin{proof}
We've explained the case $n=1$ several times already, so we assume that $n>1$ and proceed inductively.

Suppose $a\in R$ is such that $1+at\in\mult R$. If $a\in A$ then, as we noticed when introducing assumption (A3), the $3$-fold stability of $A$ supplies us with $v,w\in\mult A$ such that $1+at=(1+vt)(1+wt)$ and $1+vt\in\mult A$; so $1+wt\in\mult R$ and this completes the proof in this case. It remains to treat the case that $a\notin A$. This means that $a$ has a denominator containing at least one of $t_1,\dots,t_{n-1}$; after reordering for simplicity, we assume $t_{n-1}$ occurs in the denominator of $a$.

Put $R'=R_{t_1\cdots t_{n-2}}$ ($=A$ if $n=2$), so that $R=R'_{t_{n-1}}$. We have arranged matters so that $a=t_{n-1}^{-\al}b$ for some $\al>0$ and some $b\in R'$ which is not divisible by $t_{n-1}$. Put \[u:=t_{n-1}^\al+bt\in\mult R,\] and note that $u$ is also a unit in $R'_{t_{n-1}R'}$ (since $\al>0$ and $bt$ is not divisible by $t_{n-1}$ in $R'$); as usual, since $t_{n-1}$ is a prime element of $R'$, which is normal, this implies that $u\in\mult{R'}$.

It follows that $t_{n-1}$ mod $tR'$ is a unit in $R'/tR'$. Exactly as we argued in the previous theorem to prove (A4), this means that $\res t_{n-1}=u\res t_1^{\al_1}\cdots\res t_{n-2}^{\al_{n-2}}$ in $R'/tR'$ for some $u\in(A/tA)^\times$ and $\al_1,\dots,\al_{n-2}\in\bb Z$, where we write $\overline{\phantom{t_1}}$ to denote images mod $t$. But, also as we noted when proving (A4) above, $\res t_1,\dots,\res t_{n-1}$ are prime elements of $A/tA$; hence $\res t_{n-1}$ is associated to $\res t_j$ for some $j=1,\dots,n-2$.

Therefore there exists $c\in A$ and $w\in\mult A$ such that \[t_{n-1}+ct=wt_j\tag{\dag}.\] Examining this equation modulo $t_{n-1}$, and using that $t,t_j$ are prime elements of $A/t_{n-1}A$, we deduce that $t$ mod $t_{n-1}$ and $t_j$ mod $t_{n-1}$ are associated and, more importantly, that $c$ mod $t_{n-1}$ is a unit in $A/t_{n-1}A$. Since $t_{n-1}$ is in the Jacobson radical of $A$, this implies $c$ was already a unit of $A$. Having formula (\dag) and the knowledge that $c\in\mult A$, the rest of the proof is less obtuse.

From (\dag) and the formula for $u$, we have
\begin{align*}
1+at
	&=t_{n-1}^{-\al} u\\
	&=\left(w^{-1}t_j^{-1}\left(1+\frac{c}{t_{n-1}}t\right)\right)^\al u\\
	&=\left(1+\frac{c}{t_{n-1}}t\right)^\al v
\end{align*}
where $v:=w^{-\al}ut_j^{-\al}\in\mult{R'}$. So $1+\frac{c}{t_{n-1}}t$ is also in $\mult R$, and we note that $\frac{c}{t_{n-1}}\in\mult R$. Moreover, it follows that $v\in(1+tR)^\times$; hence $v\in(1+tR)^\times\cap\mult{R'}=(1+tR')^\times$, and the inductive hypothesis completes the proof.
\end{proof}

\begin{remark}
I do not know whether the previous lemma remains valid for more general localisations of local rings. It would seem to offer a useful tool in the study of Milnor $K$-theory, especially when combined with the Dennis-Stein-Suslin-Yarosh map.
\end{remark}

\begin{remark}
To see the necessity of some regularity hypotheses for the main theorem to be valid, consider the following situation: let $A$ be a local domain whose residue field has $>5$ elements, and recall from example \ref{example_the_ses}(ii) that the sequence \[0\to K_2^M(A[[t]])\to K_2^M(A((t))\to\mult A\to 0\] is exact. There is an analogous complex in Quillen $K$-theory which results from the fundamental sequence for Laurent polynomials. Namely, \[K_2(A[[t]])\into K_2(A((t)))\onto\mult A\] is a complex and $K_2(A((t)))$ decomposes as an direct sum $K_2(A((t)))=K_2(A[[t]])\oplus \mult A\oplus NK_2(A)$, where $NK_2(A)\cong\op{coker}(K_2(A)\to K_2(A[X]))$ (see \cite{Weibel1980} for a proof).

Comparing the sequences for Milnor and Quillen $K_2$, we see that $K_2^M(A((t)))\to K_2(A((t)))$ is an isomorphism if and only if $NK_2(A)=0$, which is true if $A$ is regular but is not true in general.
\end{remark}

\bibliographystyle{acm}
\bibliography{../Bibliography}

\begin{thebibliography}{10}

\bibitem{Dennis1973}
{\sc Dennis, R.~K., and Stein, M.~R.}
\newblock The functor {$K\sb{2}$}: a survey of computations and problems.
\newblock In {\em Algebraic {$K$}-theory, {II}: ``{C}lassical'' algebraic
  {$K$}-theory and connections with arithmetic ({P}roc. {C}onf., {S}eattle
  {R}es. {C}enter, {B}attelle {M}emorial {I}nst., 1972)}. Springer, Berlin,
  1973, pp.~243--280. Lecture Notes in Math., Vol. 342.

\bibitem{Dennis1974}
{\sc Dennis, R.~K., and Stein, M.~R.}
\newblock Injective stability for {$K\sb{2}$} of local rings.
\newblock {\em Bull. Amer. Math. Soc. 80\/} (1974), 1010--1013.

\bibitem{Dennis1975}
{\sc Dennis, R.~K., and Stein, M.~R.}
\newblock {$K\sb{2}$} of discrete valuation rings.
\newblock {\em Advances in Math. 18}, 2 (1975), 182--238.

\bibitem{Kerz2009a}
{\sc Kerz, M.}
\newblock The {G}ersten conjecture for {M}ilnor {$K$}-theory.
\newblock {\em Invent. Math. 175}, 1 (2009), 1--33.

\bibitem{Milnor1970}
{\sc Milnor, J.}
\newblock Algebraic {$K$}-theory and quadratic forms.
\newblock {\em Invent. Math. 9\/} (1970), 318--344.

\bibitem{Stein1973}
{\sc Stein, M.~R.}
\newblock Surjective stability in dimension {$0$} for {$K\sb{2}$} and related
  functors.
\newblock {\em Trans. Amer. Math. Soc. 178\/} (1973), 165--191.

\bibitem{Dennis1973a}
{\sc Stein, M.~R., and Dennis, R.~K.}
\newblock {$K_{2}$} of radical ideals and semi-local rings revisited.
\newblock In {\em Algebraic {$K$}-theory, {II}: ``{C}lassical'' algebraic
  {$K$}-theory and connections with arithmetic ({P}roc. {C}onf., {B}attelle
  {M}emorial {I}nst., {S}eattle, {W}ash., 1972)}. Springer, Berlin, 1973,
  pp.~281--303. Lecture Notes in Math. Vol. 342.

\bibitem{Suslin1991}
{\sc Suslin, A.~A., and Yarosh, V.~A.}
\newblock Milnor's {$K_3$} of a discrete valuation ring.
\newblock In {\em Algebraic {$K$}-theory}, vol.~4 of {\em Adv. Soviet Math.}
  Amer. Math. Soc., Providence, RI, 1991, pp.~155--170.

\bibitem{vanderKallen1971}
{\sc van~der Kallen, W.}
\newblock Le {$K_{2}$} des nombres duaux.
\newblock {\em C. R. Acad. Sci. Paris S\'er. A-B 273\/} (1971), A1204--A1207.

\bibitem{vanderKallen1976}
{\sc van~der Kallen, W.}
\newblock The {$K_{2}$}'s of a {$2$}-dimensional regular local ring and its
  quotient field.
\newblock {\em Comm. Algebra 4}, 7 (1976), 677--679.

\bibitem{vanderKallen1977}
{\sc van~der Kallen, W.}
\newblock The {$K\sb{2}$} of rings with many units.
\newblock {\em Ann. Sci. \'Ecole Norm. Sup. (4) 10}, 4 (1977), 473--515.

\bibitem{vanderKallen1975}
{\sc van~der Kallen, W., Maazen, H., and Stienstra, J.}
\newblock A presentation for some {$K\sb{2}(n,R)$}.
\newblock {\em Bull. Amer. Math. Soc. 81}, 5 (1975), 934--936.

\bibitem{Weibel1980}
{\sc Weibel, C.~A.}
\newblock {$K$}-theory and analytic isomorphisms.
\newblock {\em Invent. Math. 61}, 2 (1980), 177--197.

\end{thebibliography}

\noindent Matthew Morrow,\\
University of Chicago,\\
5734 S. University Ave.,\\
Chicago,\\
IL, 60637,\\
USA\\
{\tt mmorrow@math.uchicago.edu}\\
\url{http://math.uchicago.edu/~mmorrow/}\\
\end{document}